\documentclass[12pt]{amsart}
\usepackage{a4}
\usepackage{amsfonts, amssymb}
\usepackage[all]{xy}
\newtheorem{thm}{Theorem}
\newtheorem{cor}{Corollary}

\newtheorem{lem}{Lemma}

\newtheorem{defn}{Definition}
\newtheorem{exam}{Example}

\numberwithin{equation}{section} \numberwithin{thm}{section}
\numberwithin{lem}{section} \numberwithin{problem}{section}
\numberwithin{cor}{section}

\newcommand{\A}{\mathcal{A}}

\newcommand{\Z}{\mathbb{Z}}

\newcommand{\F}{\mathbb F}

\begin{document}
\title{Combinatorial problems in finite fields and Sidon sets
 }
%\subjclass{2000 Mathematics Subject Classification: 11B83.}
%\keywords{Sidon sets, $B_2[g]$ sequences, Probabilistic method}
\author{Javier Cilleruelo}
%\thanks{This work was supported by Grant MTM 2005-04730 of MYCIT (Spain)}
%\small Departamento de Matem{\' a}ticas\\[-0.8ex]
%\small Universidad Aut{\' o}noma de Madrid, 28049 Madrid, Spain \\[-0.8ex]
%\small \texttt{franciscojavier.cilleruelo@uam.es}}
%\date{\dateline{Jan 1, 2006}{Jan 2, 2006}\\
%\small Mathematics Subject Classifications: 05C88, 05C89}
\address{Instituto de Ciencias Matem\'aticas (CSIC-UAM-UC3M-UCM) and
Departamento de Matem\'aticas\\
Universidad Aut\'onoma de Madrid\\
28049, Madrid, Espa\~na} \email{franciscojavier.cilleruelo@uam.es}
\begin{abstract}
We use Sidon sets to present an elementary method to study  some
combinatorial problems in finite fields, such as sum product
estimates, solubility of some equations and distribution of
sequences in small intervals. We obtain classic and more recent
results avoiding the use of exponential sums, the usual tool to deal
with these problems.
\end{abstract}
\maketitle
\section{Introduction}
The aim of the present work is to introduce a new elementary  method
to study a class of combinatorial problems in finite fields:
incidence problems, sum-product estimates, solubility of some
equations, distribution of sequences in small intervals, etc.

The main tool in our approach are Sidon sets, which are important
objects in combinatorial number theory.

In Section \S 2 we present Theorem \ref{BB}, which is a new result
about Sidon sets and the main tool in our method.  To illustrate how
this method works, we include in this section two easy application
of this theorem. The first one recovers a result of Vinh \cite{V}
about the number of incidences between $P$ points and $L$ lines in a
field $\F_q$. The second one proves that if $A=\{g^x:\ 0\le x\le
(\sqrt 2+o(1))p^{3/4}\}$ and $g$ is a primitive root modulo $p$,
then $A-A=\Z_p$. This improves previous results obtained by
Garaev-Kueh\cite{Ga3}, Konyagin\cite{K} and Garc\'{\i}a\cite{Garcia}.

Section \S 3 is devoted to sum-product estimates. Garaev \cite{Ga1}
used character sums to give the nontrivial lower estimate
$\max(|A+A|,|AA|)\gg \min (\sqrt{|A|q},|A|^2/\sqrt q)$. Theorem
\ref{BB'}, which is an easy consequence of Theorem \ref{BB}, gives a
nontrivial upper bound for the number of elements of a dense Sidon
set in an arbitrary set $B$ when $|B+B|$ is small.  We use this
upper bound to give a quick proof of Garaev's estimate and related
results.

S\'{a}rk\H ozy \cite{Sar1,Sar2} proved the solubility of the equations
$x_1x_2+x_3x_4=1 $ and $x_1x_2=x_3+x_4,\ x_i\in A_i$ for arbitrary
sets $A_i\subset \F_p$ when $|A_1||A_2||A_3||A_4|\gg p^3$. This
result was extended to any field $\F_q$ in \cite{SG}. The proof is
based in estimates of exponential sums and they asked for an
elementary algebraic proof of the solubility of these equations
(problem 3 of \cite{CGS}). Our method provides a proof of this kind.
Actually, Sark\H ozy's equations  are special cases of more general
equations which we study in section \S 4.

In section \S 5, we apply our method to study the distribution of
some sequences in $\Z_p$. As an example we prove that if $g$ is a
primitive root modulo $p$, then for any intervals $I,J$ and any
positive integer $r$ we have
$$|\{(x,y)\in I\times J:\ g^x-g^y\equiv 1 \pmod p\}|=\frac{|I||J|}{p}+\theta_r\Bigl (\Bigl (\frac{|I||J|}{p^{3/2}}\Bigr )^{1/r}+1\Bigr )\sqrt p,$$
with $|\theta_r|\le 4^r$. The error term is smaller than the error
term obtained by Garaev \cite{Ga2}.

\section{Sidon sets}
Let $G$ be a finite abelian group. For any sets $A,B\subset G$ and
$x\in G$, we write $r_{A-B}(x)$ for the number of representations of
$x=a-b,\ a\in A,\ b\in B$. We have the well known identities
\begin{align} \sum_{x\in G}r_{A-B}(x)&=|A||B|,\label{XY}\\
\sum_{x\in G}r_{A-B}^2(x)&=\sum_{x\in G}r_{A-A}(x)r_{B-B}(x).
\label{AE}
\end{align}

\begin{defn}We say that a set $\A\subset G$ is a Sidon set if
$r_{\A- \A}(x)\le 1$ whenever $x\ne 0$.\end{defn} By counting the
number of differences $a-a'$, we can see that if $\A$ is a Sidon
set, then $|\A|<\sqrt{|G|}+1/2$. The most interesting Sidon sets are
those which have large cardinality, that is,
$|\A|=\sqrt{|G|}-\delta$  where $\delta$ is a small number. We state
our main theorem.

\begin{thm}\label{BB}Let $\A$ be a Sidon set in a finite abelian group $G$ with $|\A|=\sqrt{|G|}-\delta$. Then, for all $B,B'\subset G$
we have
\begin{equation*}
|\{(b,b')\in B\times B',\ b+b'\in
\A\}|=\frac{|\A|}{|G|}|B||B'|+\theta (|B||B'|)^{1/2}|G|^{1/4},
\end{equation*}
with $|\theta|< 1+\frac{|B|}{|G|}\max(0,\delta)$.
\end{thm}
\begin{proof}
Since $\A$ is a Sidon set,
\begin{eqnarray*}\label{Sa}\qquad \sum_{x\in G}r_{B-B}(x)r_{\A-\A}(x) &=&|\A||B|+\sum_{x\ne
0}r_{B-B}(x)r_{\A-\A}(x)\\ &\le &|\A||B|+\sum_{x\ne 0}r_{B-B}(x)
=|\A||B|+|B|^2-|B|\nonumber.\end{eqnarray*} Using this inequality
and identities (\ref{XY}) and (\ref{AE}) we have
\begin{align}\label{Sc}
\sum_{x\in G}\left ( r_{\A-B}(x)-\frac{|\A||B|}{|G|}\right )^2&=
\sum_{x\in G}r_{B-B}(x)r_{\A- \A}(x)-\frac{|\A|^2|B|^2}{|G|}\\ &\le
|B|(|\A|-1)+|B|^2\frac{|G|-|\A|^2}{|G|}.\nonumber
\end{align}
We observe that $$|\{(b,b')\in B\times B',\ b+b'\in
\A\}|-\frac{|B||B'||\A|}{|G|}=\sum_{b'\in B'}\left (
r_{\A-B}(b')-\frac{|\A||B|}{|G|}\right ).$$ Applying the
Cauchy-Schwarz inequality, taking \eqref{Sc} and
$|\A|=|G|^{1/2}-\delta$ into account  we obtain
\begin{eqnarray*}\left |\sum_{b'\in B'}\left (
r_{\A-B}(b')-\frac{|\A||B|}{|G|}\right )\right |^2&\le &|B'|\left
(|B|(|\A|-1)+|B|^2\frac{|G|-|\A|^2}{|G|}\right )\\
&=&|B'||B|\left
(|G|^{1/2}-\delta-1+|B|\frac{\delta(2|G|^{1/2}-\delta)}{|G|}\right
)\\
&<& |B||B'||G|^{1/2}\left (1 + 2\max(0,\delta)\frac{|B|}{|G|}\right
).\end{eqnarray*}
\end{proof}

The Sidon sets we will consider in applications satisfy $\delta\le
1$ and $|B|=o(|G|)$.  In these cases we have $|\theta|\le 1+o(1)$.

\subsection{Examples of dense Sidon sets}
The three families of Sidon sets we will describe next, have maximal
cardinality in their ambient group $G$. Let $g$ be a generator of
  $\F_q^*$.
\begin{exam}\label{exam1}Let $p(x),r(x)\in \F_q[X]$ be polynomials of degree $\le  2$ such that $p(x)-\mu r(x)$ is not a constant for any $\mu\in \F_q$. The set
$$\A=\{(p(x),r(x)):\ x\in \F_q\}$$
is a Sidon set in $\F_q\times \F_q$. In particular, the set
$\A=\{(x,x^2):\ x\in \F_q\}$ is a Sidon set.
\end{exam}
We have to check that when $(e_1,e_2)\ne (0,0)$ the relation
$(p(x_1),r(x_1))-(p(x_2),r(x_2))=(e_1,e_2)$ uniquely determines
$x_1$ and $x_2$.  If $p(x)$ is linear then from $p(x_1)-p(x_2)=e_1$
we obtain $x_1=x_2+\lambda$ for some $\lambda$. Thus,
$r(x_2+\lambda)-r(x_2)=e_2$ is a linear equation and we obtain $x_2$
and then $x_1$. If $p(x)$ is quadratic we consider $\mu$ such that
$p(x)-\mu r(x)$ is a linear polynomial and we proceed as above.

\begin{exam}\label{exam2}
For any generator $g$ of $\F_q^*$, the set
\begin{equation}\label{welch}\A=\{(x,g^{x}):\ x\in \Z_{q-1}\}\end{equation} is a Sidon set in
$\Z_{q-1}\times \F_q$.

Sometimes we will describe this set as $\A=\{(\log x,x):\ x\in
\F_q^*\}$ where $\log x=\log_g x$ is the discrete logarithm.
\end{exam} From $(x_1,g^{x_1})-(x_2,g^{x_2})=(e_1,e_2)\ne (0,0)$ we
have $x_1-x_2\equiv e_1\pmod{q-1}$ and hence $g^{x_1}= g^{e_1+x_2}$.
Putting this in $g^{x_1}-g^{x_2}=e_2$ we get
$g^{x_2}(g^{e_1}-1)=e_2$.

If $e_1=0$ then $e_2=0$, but we have assumed that $(e_1,e_2)\ne
(0,0)$. If $e_1\ne 0$ the last equality determines $x_2$, and then
$x_1$.

\begin{exam}\label{exam3}
For any pair of generators $g_1,g_2$ of $\F_q^*$, the set
\begin{equation}\label{golomb}
\A=\{(x,y)\in \Z_{q-1}\times \Z_{q-1}:\ g_1^x+g_2^y=1\}
\end{equation}
is a Sidon set in $\Z_{q-1}\times \Z_{q-1}$. Since translations
preserve  Sidoness property,  for any $\lambda\ne 0$, the sets
$\A=\{(x,y):\ g_1^x+g_2^y=\lambda\}$ and $\A=\{(x,y):\
g_1^x-g_2^y=\lambda\}$ are also Sidon sets. \end{exam} To see that
$\A$ is a Sidon set we have to prove that if $(e_1,e_2)\ne (0,0)$
then the equation $(x_1,y_1)-(x_2,y_2)= (e_1,e_2)$ determines
$x_1,x_2$ under the conditions $g_1^{x_1}+g_2^{y_1}=
g_1^{x_2}+g_2^{y_2}= 1 \text{ in }\F_q$. We observe that
$x_1-x_2\equiv e_1\pmod{(q-1)}$ and $y_1-y_2\equiv e_2\pmod{(q-1)}$
imply that $g_1^{x_1}=g_1^{x_2+e_1}$ and $g_2^{y_1}=g_2^{y_2+e_2}$
in $\F_q$ and we obtain $ g_1^{x_2+e_1}+g_2^{y_2+e_2}=
g_1^{x_2}+g_2^{y_2}= 1 \text{ in }\F_q $. Thus
$g_2^{y_2}(g_2^{e_2}-g_1^{e_1})=1-g_1^{e_1}$. If $e_1\ne 0$ and
$g_2^{e_2}\ne g_1^{e_1}$ we obtain $y_2$ and then $x_2,x_1$ and
$y_1$. If $e_1=0$ or $g_2^{e_2}=g_1^{e_1}$, the equation has not
solutions unless $(e_1,e_2)=(0,0)$.

\smallskip

When $q=p$ is a prime number we can identify $\F_p$ with $\Z_p$. We
ilustrate in the pictures below the three examples of Sidon sets
described above.

 {\Large
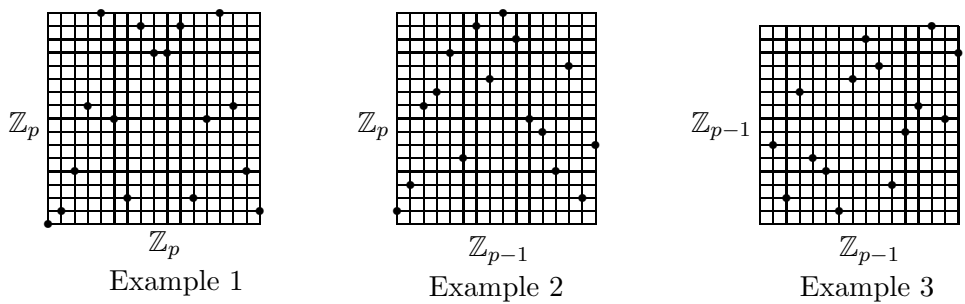
\begin{picture}(0,0)\put(-15,-5){\small $\Z_p$}\put(38,-50){\small $\Z_p$}\put(23,-65){\small Example 1}
 \put(0,-40){\line(1,0){80}}\put(0,-35){\line(1,0){80}}
 \put(0,-30){\line(1,0){80}}
\put(0,-25){\line(1,0){80}} \put(0,-20){\line(1,0){80}}
\put(0,-15){\line(1,0){80}} \put(0,-10){\line(1,0){80}}
\put(0,-5){\line(1,0){80}} \put(0,0){\line(1,0){80}}
\put(0,5){\line(1,0){80}} \put(0,10){\line(1,0){80}}
\put(0,15){\line(1,0){80}} \put(0,20){\line(1,0){80}}
\put(0,25){\line(1,0){80}} \put(0,30){\line(1,0){80}}
\put(0,35){\line(1,0){80}}\put(0,40){\line(1,0){80}}
 \put(0,-40){\line(0,1){80}}
\put(5,-40){\line(0,1){80}} \put(10,-40){\line(0,1){80}}
\put(15,-40){\line(0,1){80}} \put(20,-40){\line(0,1){80}}
\put(25,-40){\line(0,1){80}} \put(30,-40){\line(0,1){80}}
\put(35,-40){\line(0,1){80}} \put(40,-40){\line(0,1){80}}
\put(45,-40){\line(0,1){80}} \put(50,-40){\line(0,1){80}}
\put(55,-40){\line(0,1){80}} \put(60,-40){\line(0,1){80}}
\put(65,-40){\line(0,1){80}} \put(70,-40){\line(0,1){80}}
\put(75,-40){\line(0,1){80}} \put(80,-40){\line(0,1){80}}
\put(0,-40){\circle*{3}} \put(5,-35){\circle*{3}}
\put(10,-20){\circle*{3}}\put(15,5){\circle*{3}}
\put(20,40){\circle*{3}}\put(25,0){\circle*{3}}
\put(30,-30){\circle*{3}}\put(35,35){\circle*{3}}
\put(40,25){\circle*{3}}
\put(45,25){\circle*{3}}\put(50,35){\circle*{3}}
\put(55,-30){\circle*{3}}\put(60,0){\circle*{3}}
\put(65,40){\circle*{3}}\put(70,5){\circle*{3}}
\put(75,-20){\circle*{3}}\put(80,-35){\circle*{3}}
\end{picture}$\qquad \qquad \qquad \qquad $
\begin{picture}(0,40)\put(-15,-5){\small $\Z_{p}$}
\put(27,-52){\small $\Z_{p-1}$}\put(12,-67){\small Example 2}
\put(0,-40){\line(1,0){75}}\put(0,-35){\line(1,0){75}}
 \put(0,-30){\line(1,0){75}}
\put(0,-25){\line(1,0){75}} \put(0,-20){\line(1,0){75}}
\put(0,-15){\line(1,0){75}} \put(0,-10){\line(1,0){75}}
\put(0,-5){\line(1,0){75}} \put(0,0){\line(1,0){75}}
\put(0,5){\line(1,0){75}} \put(0,10){\line(1,0){75}}
\put(0,15){\line(1,0){75}} \put(0,20){\line(1,0){75}}
\put(0,25){\line(1,0){75}} \put(0,30){\line(1,0){75}}
\put(0,35){\line(1,0){75}}\put(0,40){\line(1,0){75}}
 \put(0,-40){\line(0,1){80}}
\put(5,-40){\line(0,1){80}} \put(10,-40){\line(0,1){80}}
\put(15,-40){\line(0,1){80}} \put(20,-40){\line(0,1){80}}
\put(25,-40){\line(0,1){80}} \put(30,-40){\line(0,1){80}}
\put(35,-40){\line(0,1){80}} \put(40,-40){\line(0,1){80}}
\put(45,-40){\line(0,1){80}} \put(50,-40){\line(0,1){80}}
\put(55,-40){\line(0,1){80}} \put(60,-40){\line(0,1){80}}
\put(65,-40){\line(0,1){80}} \put(70,-40){\line(0,1){80}}
\put(75,-40){\line(0,1){80}}
 \put(0,-35){\circle*{3}}
\put(5,-25){\circle*{3}}
\put(10,5){\circle*{3}}\put(15,10){\circle*{3}}
\put(20,25){\circle*{3}}\put(25,-15){\circle*{3}}
\put(30,35){\circle*{3}}\put(35,15){\circle*{3}}
\put(40,40){\circle*{3}}
\put(45,30){\circle*{3}}\put(50,0){\circle*{3}}
\put(55,-5){\circle*{3}}\put(60,-20){\circle*{3}}
\put(65,20){\circle*{3}}\put(70,-30){\circle*{3}}
\put(75,-10){\circle*{3}}\put(0,-5){\line(0,1){15}}
\end{picture}
$\qquad \qquad \qquad \qquad $
\begin{picture}(0,55) \put(-26,-5){\small
$\Z_{p-1}$} \put(30,-52){\small $\Z_{p-1}$}\put(15,-67){\small
Example 3} \put(0,-40){\line(1,0){75}}\put(0,-35){\line(1,0){75}}
 \put(0,-30){\line(1,0){75}}
\put(0,-25){\line(1,0){75}} \put(0,-20){\line(1,0){75}}
\put(0,-15){\line(1,0){75}} \put(0,-10){\line(1,0){75}}
\put(0,-5){\line(1,0){75}} \put(0,0){\line(1,0){75}}
\put(0,5){\line(1,0){75}} \put(0,10){\line(1,0){75}}
\put(0,15){\line(1,0){75}} \put(0,20){\line(1,0){75}}
\put(0,25){\line(1,0){75}} \put(0,30){\line(1,0){75}}
\put(0,35){\line(1,0){75}}
 \put(0,-40){\line(0,1){75}}
\put(5,-40){\line(0,1){75}} \put(10,-40){\line(0,1){75}}
\put(15,-40){\line(0,1){75}} \put(20,-40){\line(0,1){75}}
\put(25,-40){\line(0,1){75}} \put(30,-40){\line(0,1){75}}
\put(35,-40){\line(0,1){75}} \put(40,-40){\line(0,1){75}}
\put(45,-40){\line(0,1){75}} \put(50,-40){\line(0,1){75}}
\put(55,-40){\line(0,1){75}} \put(60,-40){\line(0,1){75}}
\put(65,-40){\line(0,1){75}} \put(70,-40){\line(0,1){75}}
\put(75,-40){\line(0,1){75}} \put(5,-10){\circle*{3}}
\put(10,-30){\circle*{3}}\put(15,10){\circle*{3}}
\put(20,-15){\circle*{3}}\put(25,-20){\circle*{3}}
\put(30,-35){\circle*{3}}\put(35,15){\circle*{3}}
\put(40,30){\circle*{3}}
\put(45,20){\circle*{3}}\put(50,-25){\circle*{3}}
\put(55,-5){\circle*{3}}\put(60,5){\circle*{3}}
\put(65,35){\circle*{3}}\put(70,0){\circle*{3}}
\put(75,25){\circle*{3}}
\end{picture}
}

\

\

\

\

 The Sidon sets given in examples, with $q,q-1$ and $q-2$ elements
 respectively, have  maximal cardinality  in their ambient
 groups. The values of $\delta=|G|^{1/2}-|\A|$ are $\delta=0,1/2-o(1)$ and $1$
 respectively. We finish this section with two easy applications of
 Theorem \ref{BB}.

 \subsection{Incidence of lines and points in $\F_q\times \F_q$}
Let $I(P,L)=|\{(p,l)\in P\times L:\ p\in
 L\}|$ be the number of incidences between a set $P$ of points and set $L$ of lines in $\F_q\times
 \F_q$, that is
 $$I(P,L)=\left |\{(p,l)\in P\times L:\ p\in l\}\right |.$$
Vinh \cite{V} proved that $I(P,L)\le
\frac{|P||L|}p+O(\sqrt{|P||L|p}).$ We recover this result as a
straightforward consequence of Theorem \ref{BB}.
\begin{thm}Let $L$ be a set of lines and let $P$ be a set of points in $\F_q\times \F_q$.
The following  asymptotic formula holds:
 \begin{equation}\label{in}I(P,L)=\frac{|P||L|}q+O(\sqrt{|P||L|q}).\end{equation}
\end{thm}
\begin{proof}
Let $L=\{y=\lambda_ix+\mu_i:\ 1\le i\le |L|\}$ and $P=\{(p_j,q_j):\
1\le j\le |P|\}$. We consider the set $\A=\{(\log x,x)\}$ described
in \eqref{exam2}
 and the sets
$$B=\{(\log \lambda_i,-\mu_i):\ 1\le i\le |L|\},\qquad B'=\{(\log p_j,
q_j):\ 1\le j\le |P|\}.$$ We observe that each incidence corresponds
to a solution of $\lambda_i p_j=q_j-\mu_i$ and the number of
solutions of this equation is $|\{(b,b')\in B\times B':\ b+b'\in
\A\}|$. The result follows in view of  Theorem \ref{BB}.
\end{proof}

\subsection{The difference set $\{g^x-g^y:\ 0,\le x,y\le L\}$} Let $g$
be a primite root modulo $p$. Many authors have studied the problem
of determining the smallest number $M$ such that $\{g^x-g^y:\ 0\le
x,y\le M\}=\Z_p$.
%Odlyzco  has conjectured that it is posible to
%take $M\ll p^{1/2+\epsilon}$.

From the result of Rudnick and Zaharescu \cite{RZ} it follows that
one can take any integer  $M\ge c_0p^{3/4}\log p$ where $c_0$ is a
suitable constant. This range has been improved to $M
> cp^{3/4}$ by  Garaev and Kueh \cite{Ga3} and
independently by Konyagin \cite{K}.
%Although the exponent $3/4$ is a
%barrier point in this problem,
The best known admissible value for the constant $c$ has been
$c=2^{5/4}$ due to Garc\'{\i}a \cite{Garcia}. Our approach improves this
further to the following statement.
\begin{thm}
Let $g$ be a generator of $\F_p^*$. For any $\epsilon>0$ and
$p>p(\epsilon)$ we have
$$\left \{ g^x-g^y:\ 0\le x,y< (\sqrt 2+\epsilon) p^{3/4}\right
\}=\F_p.$$
\end{thm}
\begin{proof}
Suppose $\lambda\not \in \{g^x-g^y:\ 0\le x,y\le L\}$ and consider
in $G=\Z_{p-1}\times \Z_{p-1}$ the Sidon set $\A=\{(x,y):\
g^x-g^y=\lambda\}$. We observe that $$(x,y)\in \A \iff
(y,x)+(\frac{p-1}2,\frac{p-1}2)\in \A.$$ Thus, it is clear that if
$b,b'\in B=[0,L/2]^2+\left \{(0,0),\left
(\frac{p-1}2,\frac{p-1}2\right )\right \} $ then
 $b+b'\not \in \A$.
%$$\A\cap \left ([0,L]^2+\left \{(0,0),\left (\frac{p-1}2,\frac{p-1}2\right )\right \}\right
%)=\emptyset.$$
 In other words,
$|\{(b,b')\in B\times B,\ b+b'\in \A\}|=0$.

On the other hand, Theorem \ref{BB} implies that
$$0=|\{(b,b')\in B\times B,\ b+b'\in \A\}|\ge
\frac{|\A||B|^2}{|G|}-\Bigl (1+\frac{|B|}{|G|}\Bigr )|B||G|^{1/4}.$$
Thus $$ |B|\le
\frac{|G|^{5/4}}{|\A|-|G|^{1/4}}=\frac{(p-1)^{5/2}}{p-2-(p-1)^{1/2}}<p^{3/2}(1+o(1))$$
and the theorem follows in view of $|B|=2(1+[L/2])^2$.

\end{proof}

\section{Sum-product estimates}
We will deduce some sum-product estimates form the following lemma.
\begin{lem}\label{BB'}
Let $\A$ be a Sidon set in $G$ with $|\A|=|G|^{1/2}-\delta$. For any
subsets  $B,B'\subset G$ we have
\begin{equation*}
|\A\cap B|\le \frac{|B+B'||\A|}{|G|}+\theta\left
(\frac{|B+B'|}{|B'|}\right )^{1/2}|G|^{1/4},
\end{equation*}
for some $\theta$ with $|\theta|\le
1+\max(0,\delta)\frac{|B'|}{|G|}$.
\end{lem}
\begin{proof}
Indeed, by Theorem \ref{BB},
\begin{align*}|B'||\A\cap B|&=|\{(-b',b+b'):\ b\in B,\quad b'\in
 B',\ -b'+(b+b')\in \A\}|\\ &\le \{(b',b''):\ b'\in
 (-B')\times (B+B'),\ b'+b''\in \A\}|\\ &\le
 \frac{|\A||B'||B+B'|}{|G|}+\theta\sqrt{|B'||B+B'|}|G|^{1/4}\end{align*}
where $|\theta|\le 1+\max(0,\delta)\frac{|B'|}{|G|}$. The lemma
follows.
 \end{proof}
%Theorem above gives quick proofs of some  sum-product estimates
%obtained in recent years.
\begin{thm}[Garaev \cite{Ga1}]\label{GGG}Let $A_1,A_2\subset \F_q^*$ and $A_3\subset \F_q$. We have
\begin{equation}\label{GG}\max(|A_1A_2|,|A_1+A_3|)\gg \min \left
(\sqrt{|A_1|q},\sqrt{|A_1|^2|A_2||A_3|/q}\right ).\end{equation}
\end{thm}
\begin{proof}
We consider the Sidon set  $\A=\{(\log x,x):\ x\in \F_q^*\}$
described in example \ref{exam2} and the sets $B=(\log A_1) \times
A_1$ and $B'=(\log A_2)\times A_3$. Since all the elements $(\log
a_1,a_1)$ are in $\A$ we have that $|\A \cap B|=|A_1|$. On the other
hand we observe that $|B+B'|=|A_1A_2||A_1+A_3|$. Lemma \ref{BB'}
implies the inequality
\begin{equation*}\label{222}|A_1|\le\frac{|A_1A_2||A_1+A_3|}q+\theta
\sqrt{q\frac{|A_1A_1||A_1+A_3|}{|A_2||A_3|}},\qquad |\theta|\le
1,\end{equation*} which in turn implies \eqref{GG}.
\end{proof}

We can mimic this proof  to get the following sum-product estimates.
\begin{thm}[Garaev-Shen \cite{GaS}]Let $A_1,A_2,A_3\subset \F_q^*$. We have
$$\max(|(A_1+1)A_2|,|A_1A_3|)\gg \min \left (\sqrt{|A_1|q},\sqrt{|A_1|^2|A_2||A_3|/q}\right ).$$
\end{thm}
\begin{proof}
We consider the Sidon set $\A=\{(x,y):\ g^x-g^y=1\}$, the sets
$B=\log(A_1+1)\times \log A_1$ and $B'=\log A_2\times \log A_3$ and
proceed as in the  proof of Theorem \ref{GGG}.
\end{proof}
\begin{thm}[Solymosi \cite{So}, Hart-Li-Shen \cite{H}]Let $p(x),q(x)\in \F_q[X]$ be polynomials of degree $\le 2$ such that $p(x)-\mu q(x)$ is not
a constant for any $\mu\in \F_q$. For any $A_1,A_2,A_3\subset \F_q$
we have
$$\max(|p(A_1)+A_2|,|q(A_1)+A_3|)\gg \min \left (\sqrt{|A_1|q},\sqrt{|A_1|^2|A_2||A_3|/q}\right ).$$
\end{thm}
\begin{proof}
We consider the Sidon set $\A=\{(p(x),q(x)):\ x\in \F_q\},$ the sets
$B=p(A_1)\times q(A_1)$ and $B'=A_2\times A_3$ and proceed as in the
proof of Theorem \ref{GGG}.
\end{proof}

Solymosi \cite{So} proved that if $\{(x,f(x)):\ x\in \F_q\}\subset
\F_q\times \F_q$ is a Sidon set then $\max (|A+A|,|f(A)+f(A)|)\gg
\min (\sqrt{|A|q},|A|^2/\sqrt q)$.
%, it can be proved that this set is
%a Sidon set if and only if $f(x)$ is a quadratic polynomial.

\section{Equations in $\F_q$}
 We start with
the easiest example which, however, we have not seen in the
literature.
\begin{thm}\label{xxx}For any $x\in \F_q$, let $X(x),Y(x)$ be any  pair of subsets of $\F_q$ and put $T=\Bigl (\sum_x|X(x)|\Bigr)\Bigl (\sum_x|Y(x)|\Bigr )$.
Then, the number of solutions $S$ of
\begin{align*}x'+y'=(x+y)^2,\quad x'\in X(x),\ y'\in Y(y)\end{align*}
is
\begin{equation*}
S=\frac{T}q+\theta\sqrt{qT}
\end{equation*}
for some $\theta$ with $|\theta|\le 1$.
\end{thm}
\begin{proof}
We consider the Sidon set $\A=\{(x,x^2):\ x\in \F_q\}$ and the sets
$$B=\{(x,x'):\ x'\in X(x)\},\qquad B'=\{(y,y'):\ y'\in Y(y)\}.$$ From the definition, $(x,x')+(y,y')\in \A \iff x'+y'=(x+y)^2$. Thus
$S=|\{(b,b')\in B\times B':\ b+b'\in \A\}|$ and we apply Theorem
\ref{BB}.
\end{proof}
\begin{cor}\label{co}Let $A_1,A_2,A_3,A_4\subset \F_q$.
Then, the number of solutions of  the equation
\begin{equation}\label{34}x_1+x_2=(x_3+x_4)^2,\qquad x_i\in A_i\end{equation} is
\begin{equation*}
S=\frac{|A_1||A_2||A_3||A_4|}q+\theta\sqrt{q|A_1||A_2||A_3||A_4|},\qquad
|\theta|\le 1.
\end{equation*}
In particular, the number of solutions of
\begin{equation}\label{zz}x_1+x_2=z^2,\quad x_1\in A_1,\ x_2\in A_2,\ z\in \F_q\end{equation} is
\begin{equation*}\label{111}|A_1||A_2|+\theta \sqrt{|A_1||A_2|q},\end{equation*} for some $\theta$ with $|\theta|\le
1$.
\end{cor}
\begin{proof}
The first part of the statement follows from Theorem \ref{xxx} by taking $$X(x)=\begin{cases}A_1,\quad x\in A_3\\
\emptyset \quad \text{ otherwise}\end{cases}
\qquad  \text{and} \qquad Y(x)=\begin{cases}A_2,\quad x\in A_4\\
\emptyset \quad \text{ otherwise}.\end{cases}$$ The second part of
the statement follows from the fact that if $A_3=A_4=\F_q$ then each
solution of  \eqref{zz} corresponds to exactly $q$ solutions of
\eqref{34}.
\end{proof}
Shkredov \cite{S2} used Weil's bound for exponential sums with
multiplicative characters to prove the following result for $q=p$
prime and the condition $|X_1||X_2|>20 p$.

\begin{cor}Let $X_1,X_2\subset \F_q$,  $|X_1||X_2|>2q$. Then
there exist $x,y\in \F_q$ such that $x+y\in X_1$ and $xy\in X_2$.
\end{cor}
\begin{proof}The number of such that pairs $(x,y)$ is equal to the number of solutions of the equation
$$ (x_1/2-z)(x_1/2+z)=x_2,\qquad x_1\in
X_1,\ x_2\in X_2,\ z\in \F_q.$$ We observe that this equation is
equivalent to the equation $(x_1/2)^2-x_2=z^2$. In order to apply
\eqref{zz} we split $X_1=X_{11}\cup X_{12}$ in such a way that the
squares in each set are all distinct. Then we apply  \eqref{111}
separately to $A_1=\{x_1^2/2:\ x_i\in X_{11}\},\ A_2=-X_2$ and to
$A_1=\{x_1^2/2:\ x_i\in X_{12}\},\ A_2=-X_2$. It follows  that the
number of solutions of the equation $(x_1/2)^2-x_2=z^2,\quad x_1\in
X_1,\ x_2\in X_2,\ z\in \F_q$ is positive in view of
$$|X_{11}||X_2|-\sqrt{|X_{11}||X_2|q}+|X_{12}||X_2|-\sqrt{|X_{12}||X_2|q}\ge |X_1||X_2|-\sqrt{2|X_1||X_2|q}>0.$$
\end{proof}

S\'{a}rk\"{o}zy \cite{Sar1,Sar2} using exponential sums, obtained asymptotic
formula for  the number of solutions of the congruences
$x_1x_2-x_3x_4\equiv \lambda \pmod p$ and $x_1x_2-x_3-x_4\equiv
\lambda \pmod p,\ x_i\in X_i$. In \cite{SG} these results have been
proved  in any finite fields. We derive Sark\H ozy's results
directly from our Theorem \ref{BB}.

\begin{thm}\label{x'+y'}For  $x\in \F_q^*,\ y\in \F_q^*$ let $X(x),Y(y)$ be any   subsets of $\F_q$.
Then for the number $S$ of solutions  of  the equation
$$x'+y'=xy,\quad x'\in X(x),\ y'\in Y(y)$$ we have
\begin{equation*}
S=\frac{T}q+\theta\sqrt{qT},\qquad |\theta|\le 1+o(1),
\end{equation*}
where $T=\Bigl (\sum_x|X(x)|\Bigr)\Bigl (\sum_x|Y(x)|\Bigr )$.
\end{thm}
\begin{proof}
We consider the Sidon set $\A=\{(x,g^x):\ x\in \Z_{q-1}\}$ and the
sets $$B=\{(\log x,x'):\ x'\in X(x)\},\qquad B'=\{(\log y,y'):\
y'\in Y(y)\}.$$ We observe that $$(\log x,x')+(\log y,y')\in \A\iff
x'+y'=g^{\log x+\log y}=xy.$$ Thus $S=|\{(b,b')\in B\times B':\
b+b'\in \A\}|$ and then we apply Theorem \ref{BB}.
\end{proof}
\begin{cor}
Let $X_1,X_2\subset \F_q^*$ and $X_3,X_4\subset \F_q$. The number
$S$ of solutions of the equation $$x_1x_2=x_3+x_4,\quad x_i\in
X_i,$$ is
$$S= \frac{|X_1||X_2||X_3||X_4|}{q}+\theta\sqrt{|X_1||X_2||X_3||X_4|q},\qquad |\theta|\le 1+o(1).$$
\end{cor}
\begin{proof}
We take $X(x)$ and $Y(y)$ as in Corollary \ref{co} and use Theorem
\ref{x'+y'}.
\end{proof}

\begin{cor}\label{xx-xx}
Let $X_1,X_2\subset \F_q^*$ and $X_3,X_4\subset \F_q$. The number
$S$ of solutions  of the equation $$x_2x_3-x_1x_4=1,\quad x_i\in
X_i,$$ is
$$S= \frac{|X_1||X_2||X_3||X_4|}{q}+\theta\sqrt{|X_1||X_2||X_3||X_4|q},\qquad |\theta|\le 1+o(1).$$
\end{cor}
\begin{proof}
In Theorem \ref{x'+y'} we take $$X(x)=\begin{cases} xX_3,\ x\in X_1^{-1}\\
\emptyset \quad \text{ otherwise}\end{cases}\qquad \text{and}\qquad
Y(y)=\begin{cases} -yX_4,\ y\in X_2^{-1}\\ \emptyset \quad \text{
otherwise}\end{cases}.$$  In this way we arrive at the equation
$x_1^{-1}x_2^{-1}=x_1^{-1}x_3-x_2^{-1}x_4$, which is equivalent to
the equation of the corollary.
\end{proof}
\begin{thm}
For  $x\in \F_q^*,\ y\in \F_q^*$, let  $X(x),Y(y)$ be any subsets of
$\F_q^*$. The number $S$ of solutions of the equation
$$xy-x'y'=1,\qquad x,y\in \F_q^*,\ x'\in X(x),\ y'\in Y(y),$$ is
$$S= \frac{T}q+\theta\sqrt{Tq},\qquad |\theta|\le 1+o(1), $$
where $T=\Bigl (\sum_x|X(x)|\Bigr)\Bigl (\sum_x|Y(x)|\Bigr )$.
\end{thm}
\begin{proof}
We consider the Sidon set $\A=\{(x,y):g^x-g^y=1\}\subset
\Z_{q-1}\times \Z_{q-1}$ and the sets $B=\{(\log x,\log x'):\ x'\in
X(x)\}$ and $\ B'=\{(\log y,\log y'):\ y'\in Y(y)\}$. It is clear
that $S=|\{(b,b')\in B\times B':\ b+b'\in \A\}|$. Now we apply
Theorem \ref{BB}.
\end{proof}
We observe that this theorem also gives an alternative proof of
Corollary \ref{xx-xx} by taking $X(x)$ and $Y(y)$ as in Corollary
\ref{co}.

\section{Distribution of Sidon sets  and
applications} Let $\A$ be a Sidon set in $G$. For any set $B\subset
G$ we write $E_{\A}(B)=|\A\cap B|-\frac{|B||\A|}{|G|}$.

The following lemma and Lemma \ref{BB'} will be the main tools to
prove asymptotic estimates for $|\A\cap B|$ in a class of problems.
For simplicity we restrict ourselves to the cases when $\A$ is one
of the three Sidon sets described in Section 2.

\begin{lem}\label{E}Let $\A$ be one of the three Sidon sets described in section \S2 and $B\subset G$. For any set $C\subset G$, there exists $c\in C$
such that
$$|E_{\A}(B)|\le 2\Bigl (q\frac{|B|}{|C|}\Bigr )^{1/2}+|E_{\A}(B^c)|+|E_{\A}(B_c)|$$
where $B^c=B\setminus (B+c)$ and $B_c=(B+c)\setminus B$.
\end{lem}
\begin{proof}We have
\begin{align}\label{ss}
E_{\A}(B)=|\A \cap B|-\frac{|\A||B|}{|G|}=\frac 1{|C|}&\sum_{c\in
C}\Bigl (|\A\cap (B+c)|-\frac{|\A||B|}{|G|}\Bigr )\nonumber\\+\frac
1{|C|}&\sum_{c\in C}\Bigl (|\A \cap B|-|\A \cap (B+c)|\Bigr ).
\end{align}
We observe that \begin{equation}\label{first}\sum_{c\in C}\Bigl
(|\A\cap (B+c)|-\frac{|\A||B|}{|G|}\Bigr )=|\{(b,c)\in B\times C:\
b+c\in \A\}|-\frac{|\A||B||C|}{|G|}.\end{equation}

Hence, by Theorem \ref{BB}, the absolute value of this sum is
bounded
 by $2\Bigl (q|B||C|\Bigr )^{1/2}$.

Since $|B_c|=|B^c|$, for the second sum in \eqref{ss} we have
\begin{align*}|\A \cap B|-|\A \cap (B+c)|&=|\A \cap B_c|-|\A\cap
B^c|\\&=\left ( |\A \cap B_c|- \frac{|\A||B_c|}{|G|}\right )-\left
(|\A\cap B^c|- \frac{|\A||B^c|}{|G|}\right
)\\&=E_{\A}(B_c)-E_{\A}(B^c).
\end{align*}  Thus
\begin{align*}
\left |\frac 1{|C|}\sum_{c\in C}\Bigl (|\A \cap B|-|\A \cap
(B+c)|\Bigr )\right |\le  \max_{c\in C}\Bigl
(|E_{\A}(B_c)|+|E_{\A}(B^c)|\Bigr ).
\end{align*}
\end{proof}
In the special case when $B$ is a subgroup we can take $C=B$ and
then $B^c=B_c=\emptyset $ for any $c\in C$. Thus, in this case we
have
$$|E_{\A}(B)|\ll q^{1/2}.$$
As a corollary we obtain a well known result on the Fermat equation
in finite fields.
\begin{cor}Let $Q,Q'$  be subgroups of $\F_q^*$. We have
$$|\{(x,y)\in Q\times Q':\ x+y=1\}|=\frac{|Q||Q'|}q+O(\sqrt q).$$
In particular, if $p\gg (rs)^2$ the Fermat congruence $x^r+y^s\equiv
1\pmod p$ has nontrivial solutions.
\end{cor}
\begin{proof}
Consider the Sidon set $\A=\{(x,y):\ g^x+g^y=1\}$ and take
$B=C=Q\times Q'$.
\end{proof}

In applications, the strategy is to take a large set $C$  such that
$|B^c|$ and $|B_c|$ are small compared with $|B|$. This is possible
when $B$ has some specific regularity properties (subgroups,
cartesian product of arithmetic progressions, convex sets, etc.) We
illustrate our method with an example.
\begin{thm}\label{rr}Let $I,J\subset \Z_{p-1}$ be intervals.
For any positive integer $r$ we have
$$\{(x,y)\in I\times J:\ g^x-g^y\equiv \lambda \pmod p\}=\frac{|I||J|}{p}+4^r\theta\Bigl (\Bigl (\frac{|I||J|}{p^{3/2}}\Bigr )^{1/r}+1\Bigr )\sqrt p,$$
with $|\theta|\le 1$.
\end{thm}
\begin{proof}We proceed by induction on $r$. We consider the Sidon set $\A=\{(x,y):\ g^x-g^y=\lambda
\}\subset \Z_{p-1}\times \Z_{p-1}$ and the set $B=I\times J$. Then,
applying Lemma \ref{BB'}, we get
$$|\A\cap B|\le \frac{|B+B|}{p}+2\sqrt{p\frac{|B+B|}{|B|}}\le
\frac{4|I||J|}p+4\sqrt p.$$ Since $|E_{\A}(B)|\le
\max(\frac{|B||\A|}{|G|},|\A\cap B|)$, we have that $|E_{\A}(B)|\le
\frac{4|I||J|}p+4\sqrt p$, which proves Theorem \ref{rr} for $r=1$.
Now we assume that Theorem \ref{rr} is true for some $r$ and we
proved it for $r+1$.

We consider the auxiliar set $C=I'\times J'$ where $I'=\{0,\dots,
\lfloor \alpha|I|\rfloor \}$ and $J'=\{0,\dots, \lfloor
\alpha|J|\rfloor \}$ for a suitable $\alpha$. We observe that
$|C|\ge \alpha^2|I||J|$. Lemma \ref{E} gives $|E_{\A}(B)|\le
2\frac{p^{1/2}}{\alpha}+|E_{\A}(B^c)|+|E_{\A}(B_c)|$. Now we observe
that $B+c$ is a small translation of the rectangle $B=I\times J$.
Thus we can write $B^c=B_1\cup B_2$ and $B_c=B_3\cup B_4$ where the
sets $B_i$ are rectangles with $|B_i|\le \alpha |I||J|$.

Thus, $|E_{\A}(B)|\le
2\frac{p^{1/2}}{\alpha}+|E_{\A}(B_1)|+|E_{\A}(B_2)|+|E_{\A}(B_3)|+|E_{\A}(B_4)|.$

Taking into account the induction hypothesis for each $B_i$ we have
$$|E_{\A}(B)|\le 2\frac{p^{1/2}}{\alpha}+4\theta_r\Bigl (\Bigl (\frac{\alpha
|I||J|}{p^{3/2}}\Bigr )^{1/r}+1\Bigr )\sqrt p.
$$ Taking $\alpha=\frac 1{4^r}\left (\frac{p^{3/2}}{|I||J|}\right )^{1/(r+1)}$ we obtain the required estimate.
\end{proof}

It should be mentioned that, for the particular case $|I|=|J|$,
Garaev obtained the error term $O(|I|^{2/3}\log^{2/3}
(|I|p^{-3/4}+2)+p^{1/2})$. We note that the error term in Theorem
\ref{rr} is smaller than Garaev's error term. Furthermore, in the
range $p^{3/2}\ll|I||J|\ll p^{3/2}(\log p)^{\log \log p}$ our error
term is smaller than the error term  $O(p^{1/2}\log^2 p)$
established in \cite{RZ}. For arbitrary intervals, Theorem \ref{rr}
gives $O_{\epsilon}(p^{1/2+\epsilon})$ for any $\epsilon
>0$, which is only slightly weaker than $O(p^{1/2}\log^2 p)$.

Finally, we remark that the analogy of Theorem \ref{rr} also holds
for some other problems of similar flavor, like estimating $\{x\in
I:\ x^2\in J\}$ or $\{x\in I:\ g^x\in J\}$. These are achieved by
employing suitable Sidon sets.

\

\textbf{Acknowledgments:} Parts of this work have been presented in
several places: Conference on Arithmetic Combinatorics (IAS
Princeton 2007), Primera Reuni\'{o}n Conjunta Sociedad Matem\'{a}tica
Mexicana-Real Sociedad Matem\'{a}tica Espa\~{n}ola (Oaxaca, M\'{e}xico, 2009)
and Discrete Analysis Seminar (University of Cambridge, 2010). I
would like to thank these institutions for their hospitality. Also I
want to thank Igor Shparlinsky and M. Garaev  for useful
conversations about this work.

\end{document}